\documentclass[preprint,12pt]{elsarticle}
\usepackage{subfigure}
\usepackage[english]{babel}
\usepackage{graphicx}
\usepackage{extarrows}

\usepackage{lineno}
\usepackage{indentfirst, latexsym, bm, amssymb,amsmath}
\usepackage{mathrsfs}
\usepackage[inline]{enumitem}
\usepackage{amssymb}
\usepackage{pifont}
\usepackage{tikz}
\usepackage{natbib}

\begin{document}
		\newtheorem{theorem}{Theorem}[section]
	\newtheorem{problem}[theorem]{Problem}
	\newtheorem{corollary}[theorem]{Corollary}
	\newtheorem{observation}[theorem]{Observation}
	\newtheorem{definition}[theorem]{Definition}
	\newtheorem{conjecture}[theorem]{Conjecture}
	\newtheorem{question}[theorem]{Question}
	\newtheorem{lemma}[theorem]{Lemma}
	\newtheorem{proposition}[theorem]{Proposition}
	\newtheorem{example}[theorem]{Example}
	\newenvironment{proof}{\noindent {\bf
			Proof.}}{\hfill $\square$\par\medskip}
	\newcommand{\remark}{\medskip\par\noindent {\bf Remark.~~}}
	\newcommand{\pp}{{\it p.}}
	\newcommand{\de}{\em}

	\newcommand{\1}{{\uppercase\expandafter{\romannumeral1}}}
	\newcommand{\2}{{\uppercase\expandafter{\romannumeral2}}}
	\newcommand{\3}{{\uppercase\expandafter{\romannumeral3}}}
	\newcommand{\4}{{\uppercase\expandafter{\romannumeral4}}}

	\begin{frontmatter}
	\title{Spectral radius and  the  2-power  of  Hamilton  paths\tnoteref{tnote1}}
	\tnotetext[tnote1]{This work is supported by National Natural Science
		Foundation of China (12271169), Science and Technology Commission
		of Shanghai Municipality (22DZ2229014) and Innovation
		Action Plan (Basic research projects) of Science and Technology Commission of Shanghai Municipality (21JC1401900).}
	
	
	\author[1]{Te Pi} 
	\ead{pt20050912@163.com}
	
	\author[2]{Rui Sun\corref{cor}} 
	\cortext[cor]{Corresponding author.}
	\ead{sunruicaicaicai@163.com}
	
	\author[2,3] {Long-Tu Yuan} 
	\ead{ltyuan@math.ecnu.edu.cn}
	
	\address[1]{ Shanghai Shibei Senior High School, Shanghai 200071, China.}
	\address[2] {Department of Mathematicss, East China Normal University, Shanghai 200241, China.}
    \address[3]	{Key Laboratory of MEA(Ministry of Education)  and Shanghai Key Laboratory of PMMP, Shanghai 200241, China.}
	\date{}
	
	\begin{abstract}
		
		We determine the maximum number of a graph without containing the 2-power of a Hamilton path.
		Using this result, we establish a spectral condition for a graph containing the 2-power of a Hamilton path.
	\end{abstract}

%
	
	
	\end{frontmatter}
	
	\section{Introduction}

\noindent 

Graphs considered below will always be simple. 
A simple graph $G$ consists of a finite nonempty set of vertices $V(G)$ and a set of edges $E(G)$.
Let $e(G)=\left| E(G) \right| $.
If $uv$ is an edge in graph $G$, edge $uv$ is said to be incident with vertices $u$ and $v$, and  vertices $u$ and $v$ are said to be adjacent.
Let $d(u)$ be the number of edges in $G$ which incident with vertex  $u$.
We denote by $\bigtriangleup(G)$ and $\delta(G)$ the maximum and minimum degree of $G$, respectively.
Let $\delta^*(G) =\min\{d(u): u\in V(G) ~\mbox{is a non-isolated vertex}\}$.
We use $C_n$, $P_n$, $K_n$ and $S_n$ to denote the cycle, the path, the complete graph and the star on $n$ vertices, respectively.
For a subgraph $H$ of $G$, we use $G-E(H)$ to denote the graph obtained from $G$ by deleting edges of $H$.
The complement graph of $G$, denoted $\overline{G}$, is the same vertex set as $G$, but in which two such vertices are adjacent if and only if they are not adjacent in $G$.
We call a cycle and a path contain all vertices of $G$ as a Hamilton cycle  and  a Hamilton path of $G$, respectively.
For graphs $G$ and $H$, we denote $G\cup H$ by the disjoint union of $G$ and $H$.

Throughout the paper we use the standard graph theory notation (see, e.g., \cite{2022Khan112908}).
We use $G^{+t}$ to denote the set of graphs obtained from $G$ by adding a new vertex and joining it to any $t$ vertices of $G$.
In particularly, we use $G^+$ instead of $G^{+t}$ for $t=1$.
Let $G^-$ denote the set of graphs obtained from $G$ by deleting any edge.
The 2-power of a graph $G$, denoted by $G^2$, is another graph that has the same vertex set as $G$, but in which two vertices are adjacent when their distance in $G$ is at most two.
For graphs $G$ and $H$, we say that $G$ packs with $H$ if $K_n$ contains edge-disjoint copies of $G$ and $H$.
In \cite{1961Ore315}, Ore got the maximum number of edges in a graph without containing a Hamilton cycles.
	\begin{theorem}[Ore \cite{1961Ore315}]
	Let $G$ be a graph on $n\geqslant4$ vertices. 
	If $e(G)\geqslant {n-1 \choose 2}+1$, then $G$ contains a Hamilton cycle unless $G= K_n -E(S_{n-1})$ or $G=K_5 -E(K_3)$.
\end{theorem}
Fiedler and Nikiforov \cite{2010Fiedler2170} determined the maximum number of edges in a graph without containing a Hamilton paths.
In 2022, Khan and Yuan \cite{2022Khan112908} determined the maximum number of edges of a graph without containing the 2-power of a Hamilton cycle and characterized all its extremal graphs.

We define the forbidden family of graphs $\mathcal{H}_n$ with $n\geqslant 6$ as follows (see Table 1) and let $\mathcal{H}^*_n$
be the sets of graphs obtained from $\mathcal{H}_n$ by adding $S_{n-2} \cup K_2$ and $ S_{n-1}$ to $\mathcal{H}_n$ for $n\in \{6,9\}$.
We call $G$  a $\mathcal{H}^*_n$-free graph if $G$ contains no  graph in $\mathcal{H}^*_n$ as a subgraph.
In particularly, we call $G$  a $H$-free graph instead of a $\mathcal{H}^*_n$-free graph for $\mathcal{H}^*_n = \{{H}\}$.

\begin{table}[htbp]
	\begin{center}
		\begin{tabular}{|c|c|c|c|}
			\hline
			$n$ &  $\mathcal{H}_n$ & $e(H)$, $H\in \mathcal{H}_n$ & $t=\left\lfloor {n}/ {4}\right\rfloor$\\
			\hline
			6 &  $ K_3$  & 3& 1\\
			\hline
			7 & $K_4^-, S_5\cup K_2, S_6$ & 5 & 1 \\
			\hline
			8 & $K_4, S_6\cup K_2, S_7$ & 6 & 2 \\
			\hline
			9 & $K_4$ & 6 & 2 \\
			\hline
			10 & $S_8\cup K_2, S_9$ & 8 & 2 \\
			\hline
			11 & $S_9\cup K_2, S_{10}$ & 9 & 2 \\
			\hline
			12 & $K_5, S_{10}\cup K_2, S_{11}$ & 10 & 3\\
			\hline
			13 & $S_{11}\cup K_2, S_{12}$ & 11 & 3 \\
			\hline
			$n\geqslant14$ &  $S_{n-2} \cup K_2, S_{n-1}$ & $n-2$& $\left\lfloor{n}/{4}\right\rfloor$\\
			\hline
		\end{tabular}
	\end{center}
	\centering  
	\caption{the graphs in $\mathcal{H}_n$}  
	\label{table1}  
\end{table}

We will establish the following theorem.
	\begin{theorem}\label{theorem1.1}
	Let $H$ be a graph on $n$ vertices with at most $n-2$ edges. Then $H$ packs with $P_n^2$ if and only if $H$ is $\mathcal{H}^*_n$-free graph.
\end{theorem}

As a corollary of Theorem \ref{theorem1.1}, we determine the maximum number of edges in $n$-vertex $P_n ^2$-free graphs.

\begin{corollary}\label{corollary1}
Let $G$ be a $P_n ^2$-free graph on $n \geqslant6$ vertices.
Then we have 
$$e(G) \leqslant\left\{\begin{array}{ll}12, & n=6 ; \\ 
	30, & n=9 ; \text { and } \\ {n-1 \choose 2}+1, & \text { otherwise. }\end{array}\right.$$
 Moreover, the equality holds if and only if $G=K_n-E(H)$ with $H \in \mathcal{H}_n$.
\end{corollary}

Let $A$ be the adjacency matrix of $G$. The spectral radius of $G$, denoted by $\mu(G)$, is the maximum eigenvalue of $A$.
In 2023, Yan, He, Feng and Liu \cite{2023Yan113155} established a spectral condition for a graph containing $C^2_n$.
\begin{theorem}[Yan, He, Feng and Liu \cite{2023Yan113155}\label{theorem01}]
Let $G$ be a graph on $n \geqslant 18$ vertices. If $\mu(G)>n-2$, then $G$ contains $C_n^2$ unless $G$ is a subgraph of  $K_n- E(S_{n-3})$.
\end{theorem}

We obtain the following theorem concerning $P_n^2$ and $\mu(G)$.
		\begin{theorem}\label{theorem2}
		Let $G$ be an $n$-vertex graph and $n \geqslant 6$. If $\mu(G)>n-2$, then $G$ contains $P_n^2$ unless $G$ is a subgraph of  $K_n- E(S_{n-1})$ or $K_n- E(K_3)$ for $n=6$,
		and a subgraph of $K_n- E(S_{n-1})$ for $n \geqslant 7$.
	\end{theorem}
	
	\section{Proof of Theorem \ref{theorem1.1}}\label{section 2}
	
	The proof of Theorem \ref{theorem1.1} is based on the following proposition.
	
\begin{proposition}\label{proposition2.0}
	 Let $n\geqslant 7$ and $ s \leqslant \left\lfloor{n}/{4}\right\rfloor$. If $P^2_{n-1}$ packs with $F$, then $P^2_{n}$ packs with each graph in $F^{ +s}$.
\end{proposition}
	\begin{proof}
Let $P_{n-1}=v_1\ldots v_{n-1}$.
Suppose that  $\overline{P^2_{n-1}}$ contains a copy of $F$.
For any four consecutive vertices, say $x_1,x_2,x_3,x_4$ on  $\overline{P^2_{n-1}}$, we can add  a new vertex $y$, edges $x_1x_3, x_2x_4$ and all edges between $y$ and $V(\overline{P^2_{n-1}})\setminus \{x_1,x_2,x_3,x_4\}$ to obtain $\overline{P^2_{n}}$.
If we add a new vertex $y$ and join all edges between $y$ and $V(\overline{P^2_{n-1}})\setminus \{v_1,v_2\}$ (or  $V(\overline{P^2_{n-1}})\setminus \{v_{n-1},v_{n-2}\}$), then the resulting graph is $\overline{P^2_{n}}$.
Thus if $\overline{P^2_{n}}$ is $F^\prime$-free for some $F^\prime\in F^{ +s}$, then the added vertex $z$ must adjacent to at least one vertex of $v_1,v_2$, at least one vertex of $v_{n-2},v_{n-1}$ and at least one vertex of any four consecutive vertices $\overline{P^2_{n-1}}$.
Therefore, $s\geqslant 2+  \lfloor(n-4)/4 \rfloor= \lfloor n/4 \rfloor+1$, contradicting $ s \leqslant \left\lfloor{n}/{4}\right\rfloor$.
\end{proof}
	
	For a subgraph $H$ of $G$, we use $G-H$ to denote the graph obtained from $G$ by deleting vertices and edges of $H$.
	
\begin{center}
	\begin{tikzpicture}
		\path
		(0, 0)
		\foreach \i in {0, ..., 5} {
			+(360/6 * \i:1cm) coordinate (corner \i)
		}
		;\draw [line width=0.2pt]
		(corner 1) -- (corner 4)
		(corner 2) -- (corner 5)
		(corner 3) -- (corner 0)
		(corner 3) -- (corner 5)
		(corner 4) -- (corner 0)
		(corner 4) -- (corner 5)
		;
		
		\draw [line width=1pt][red]
			(corner 0) -- (corner 4)
		    (corner 0) -- (corner 3)
		     (corner 2) -- (corner 5)
		     (corner 5) -- (corner 4);
		
		\fill[radius=2pt] \foreach \i in {0, ..., 5} { (corner \i) circle(2pt) };
		\node at (0,-1.5){$\overline{P_6^2}$};
		
		\path
		(3, 0)
		\foreach \i in {0, ..., 5} {
			+(360/6 * \i:1cm) coordinate (corner \i)
		}
		;\draw [line width=0.2pt]
		(corner 1) -- (corner 4)
		(corner 2) -- (corner 5)
		(corner 3) -- (corner 0)
		(corner 3) -- (corner 5)
		(corner 4) -- (corner 0)
		(corner 4) -- (corner 5)
		;
		\draw [line width=1pt][red]
		(corner 2) -- (corner 5)
		(corner 1) -- (corner 4)
		(corner 5) -- (corner 4)
		(corner 0) -- (corner 3);
		\fill[radius=2pt] \foreach \i in {0, ..., 5} { (corner \i) circle(2pt) };
		\node at (3,-1.5){$\overline{P_6^2}$};
		
		\path
		(6, 0)
		\foreach \i in {0, ..., 5} {
			+(360/6 * \i:1cm) coordinate (corner \i)
		}
		;
		\draw [line width=0.2pt]
		(corner 1) -- (corner 4)
		(corner 2) -- (corner 5)
		(corner 3) -- (corner 0)
		(corner 3) -- (corner 5)
		(corner 4) -- (corner 0)
		(corner 4) -- (corner 5)
		;
		\draw [line width=1pt][red]
		(corner 1) -- (corner 4)
		(corner 5) -- (corner 4)
		(corner 5) -- (corner 2)
		(corner 5) -- (corner 3)
		;
		\fill[radius=2pt] \foreach \i in {0, ..., 5} { (corner \i) circle(2pt) };
		\node at (6,-1.5){$\overline{P_{6}^2}$};
		
		\path
		(9, 0)
		\foreach \i in {0, ..., 5} {
			+(360/6 * \i:1cm) coordinate (corner \i)
		}
		;
		\draw [line width=0.2pt]
		(corner 1) -- (corner 4)
		(corner 2) -- (corner 5)
		(corner 3) -- (corner 0)
		(corner 3) -- (corner 5)
		(corner 4) -- (corner 0)
		(corner 4) -- (corner 5)
		;
		\draw [line width=1pt][red]
		(corner 5) -- (corner 3)
		(corner 2) -- (corner 5)
		(corner 1) -- (corner 4)
		(corner 0) -- (corner 4)
		;
		\fill[radius=2pt] \foreach \i in {0, ..., 5} { (corner \i) circle(2pt) };
		\node at (9,-1.5){$\overline{P_{6}^2}$};
		
		\node at (4.5,-2.5){Figure 2};
	\end{tikzpicture}
\end{center}

\begin{center}
	\begin{tikzpicture}
		\path
		(0, 0)
		\foreach \i in {0, ..., 5} {
			+(360/6 * \i:1cm) coordinate (corner \i)
		}
		;\draw [line width=0.2pt]
		(corner 1) -- (corner 4)
		(corner 2) -- (corner 5)
		(corner 3) -- (corner 0)
		(corner 3) -- (corner 5)
		(corner 4) -- (corner 0)
		(corner 4) -- (corner 5)
		;
		
		\draw [line width=1pt][red]
		(corner 0) -- (corner 4)
		(corner 0) -- (corner 3)
		(corner 5) -- (corner 4)
		(corner 3) -- (corner 5);
		
		\fill[radius=2pt] \foreach \i in {0, ..., 5} { (corner \i) circle(2pt) };
		\node at (0,-1.5){$\overline{P_6^2}$};
		
		\path
		(4, 0)
		\foreach \i in {0, ..., 6} {
			+(360/7 * \i:1cm) coordinate (corner \i)
		}
		;\draw [line width=0.2pt]
		(corner 2) -- (corner 5)
		(corner 2) -- (corner 6)
		(corner 3) -- (corner 0)
		(corner 4) -- (corner 0)
		(corner 3) -- (corner 6)
		(corner 4) -- (corner 6)
		(corner 4) -- (corner 1)
		(corner 5) -- (corner 0)
		(corner 5) -- (corner 1)
		(corner 5) -- (corner 6)
		;
		\draw [line width=1pt][red]
		(corner 2) -- (corner 5)
		(corner 2) -- (corner 6)
		(corner 5) -- (corner 6)
		;
		\fill[radius=2pt] \foreach \i in {0, ..., 6} { (corner \i) circle(2pt) };
		\node at (4,-1.5){$\overline{P_7^2}$};
		
		\path
		(8, 0)
		\foreach \i in {0, ..., 8} {
			+(360/9 * \i:1cm) coordinate (corner \i)
		}
		;
		\draw [line width=0.2pt]
		(corner 0) -- (corner 3)
		(corner 0) -- (corner 4)
		(corner 0) -- (corner 5)
		(corner 0) -- (corner 6)
		(corner 1) -- (corner 4)
		(corner 1) -- (corner 5)
		(corner 1) -- (corner 6)
		(corner 1) -- (corner 7)
		(corner 2) -- (corner 5)
		(corner 2) -- (corner 6)
		(corner 2) -- (corner 7)
		(corner 2) -- (corner 8)
		(corner 3) -- (corner 6)
		(corner 3) -- (corner 7)
		(corner 3) -- (corner 8)
		(corner 4) -- (corner 7)
		(corner 4) -- (corner 8)
		(corner 5) -- (corner 8)
		(corner 8) -- (corner 6)
		;
		\draw [line width=1pt][red]
		(corner 7) -- (corner 3)
		(corner 6) -- (corner 3)
		(corner 7) -- (corner 6)
		(corner 7) -- (corner 1)
		(corner 1) -- (corner 6)
		;
		\fill[radius=2pt] \foreach \i in {0, ..., 8} { (corner \i) circle(2pt) };
		\node at (8,-1.5){$\overline{P_{9}^2}$};
		
		\node at (4,-2.5){Figure 3};
	\end{tikzpicture}
\end{center}

\begin{center}
	\begin{tikzpicture}
		\path
		(0, 0)
		\foreach \i in {0, ..., 9} {
			+(360/10 * \i:1cm) coordinate (corner \i)
		}
		;
		\draw [line width=0.2pt]
		(corner 0) -- (corner 3)
		(corner 0) -- (corner 4)
		(corner 0) -- (corner 5)
		(corner 0) -- (corner 6)
		(corner 0) -- (corner 7)
		(corner 1) -- (corner 4)
		(corner 1) -- (corner 5)
		(corner 1) -- (corner 6)
		(corner 1) -- (corner 7)
		(corner 1) -- (corner 8)
		(corner 2) -- (corner 5)
		(corner 2) -- (corner 6)
		(corner 2) -- (corner 7)
		(corner 2) -- (corner 8)
		(corner 2) -- (corner 9)
		(corner 3) -- (corner 6)
		(corner 3) -- (corner 7)
		(corner 3) -- (corner 8)
		(corner 3) -- (corner 9)
		(corner 4) -- (corner 9)
		(corner 4) -- (corner 7)
		(corner 4) -- (corner 8)
		(corner 5) -- (corner 8)
		(corner 5) -- (corner 9)
		(corner 6) -- (corner 9)
		(corner 7) -- (corner 9)
		(corner 8) -- (corner 6)
		;
		\draw [line width=1pt][red]
		(corner 7) -- (corner 4)
		(corner 8) -- (corner 4)
		(corner 8) -- (corner 1)
		(corner 7) -- (corner 8)
		(corner 7) -- (corner 1)
		(corner 1) -- (corner 4)
		;
		\fill[radius=2pt] \foreach \i in {0, ..., 9} { (corner \i) circle(2pt) };
		\node at (0,-1.5){$\overline{P_{10}^2}$};
		
		\path
		(5, 0)
		\foreach \i in {0, ..., 9} {
			+(360/10 * \i:1cm) coordinate (corner \i)
		}
		;
		\draw [line width=0.2pt]
		(corner 0) -- (corner 3)
		(corner 0) -- (corner 4)
		(corner 0) -- (corner 5)
		(corner 0) -- (corner 6)
		(corner 0) -- (corner 7)
		(corner 1) -- (corner 4)
		(corner 1) -- (corner 5)
		(corner 1) -- (corner 6)
		(corner 1) -- (corner 7)
		(corner 1) -- (corner 8)
		(corner 2) -- (corner 5)
		(corner 2) -- (corner 6)
		(corner 2) -- (corner 7)
		(corner 2) -- (corner 8)
		(corner 2) -- (corner 9)
		(corner 3) -- (corner 6)
		(corner 3) -- (corner 7)
		(corner 3) -- (corner 8)
		(corner 3) -- (corner 9)
		(corner 4) -- (corner 9)
		(corner 4) -- (corner 7)
		(corner 4) -- (corner 8)
		(corner 5) -- (corner 8)
		(corner 5) -- (corner 9)
		(corner 6) -- (corner 9)
		(corner 7) -- (corner 9)
		(corner 8) -- (corner 6)
		;
		\draw [line width=1pt][red]
		(corner 7) -- (corner 3)
		(corner 8) -- (corner 3)
		(corner 8) -- (corner 6)
		(corner 7) -- (corner 8)
		(corner 7) -- (corner 0)
		(corner 0) -- (corner 3)
		(corner 6) -- (corner 3)
		(corner 0) -- (corner 6)
		;
		\fill[radius=2pt] \foreach \i in {0, ..., 9} { (corner \i) circle(2pt) };
		\node at (5,-1.5){$\overline{P_{10}^2}$};
		
		\node at (2.5,-2.5){Figure 4};
	\end{tikzpicture}
\end{center}

	\noindent {\bf Proof of Theorem~\ref{theorem1.1}.}
	Let $n\geqslant 6$ and $t = \left\lfloor{n}/{4}\right\rfloor$. 
	Let $F$ be an $n$-vertex graph with at most $n - 2$ edges. 
	Since $\bigtriangleup (\overline{P^2_{n}}) =n-3$, 
	$P^2_{n}$ does not pack with $S_{n-1}$.
	Note that $\overline{P^2_{n}} - S_{n-2}$ are two isolated vertices.
    So $P^2_{n}$ does not pack with  $S_{n-2}\cup K_2$.
	Assume that $F$ is $\mathcal{H}^*_n$-free graph. 
	If $n = 6$, then it is clear that $F$ packs with $P^2_{6}$ (see Figures 2 and 3).
	For $7\leqslant n \leqslant 13$, assume that the theorem holds for $n - 1$. For each $n$, we consider $F$ in the following three cases:
	\begin{itemize}
	\item (a) $\delta^*(F)\geqslant t+1$,
	
	
	\item (b) $\delta^*(F)\leqslant t$ and there is a vertex $x$ with $1 \leqslant d(x) \leqslant t$ such that $F - x$ is $\mathcal{H}^*_{n-1}$-free graph and
	
	\item (c) $\delta^*(F)\leqslant t$ and $F - x$ contains some graph in $\mathcal{H}^*_{n-1}\setminus \{S_{n-2}, S_{n-3}\cup K_2\}$ as a subgraph for each $x$ with  $1 \leqslant d(x) \leqslant t$.
	\end{itemize}
	For $\delta^*(F)\leqslant t$, if $F - x$ contains $S_{n-2}$  or $S_{n-3}\cup K_2$ as a subgraph for some vertex $x$ with $1 \leqslant d(x) \leqslant t$, then there are $n-2$ edges in $F$ and $d(x)=1$.
	Since $F$ is $\mathcal{H}^*_n$-free graph, we can easily find  a vertex $y\in V(F)$  with $1 \leqslant d(y) \leqslant t$ such that $F - y$ is $\mathcal{H}^*_{n-1}$-free graph. i.e.,  $F$ belongs to case (b). 
Therefore, $F$ belongs one of cases (a), (b) or (c).

    For all $7\leqslant n \leqslant 13$, in case (b), by the induction hypothesis, $F - x$ packs with $P^2_{n-1}$, and hence $F$ packs with $P^2_{n}$ according to Proposition \ref{proposition2.0}. Thus, we are left with cases (a) and (c).

	
	Let $n=7$. Then $t = 1$. The graphs in case (a) are $C_5$, $C_4$ and $K_3$ (see Figure 3). It is easy to see that $P^2_{7}$ packs with $C_5$, $C_4$ and $K_3$. 
Note that $\mathcal{H}^*_{6}\setminus\{S_{4}, S_{3}\cup K_2\}=\{K_3\}$.
The graphs in case (c) are $K_3 \cup P_3$, $K_3 \cup M_2$, $K^+_3\cup K_2$, $G_1$, $G_2$ and $G_3$, where $M_2$ is the 4-vertex graph on 2 independent edges,
	 $G_1$, $G_2$ and $G_3$ are obtained from $K ^+_3$ by adding a new vertex and connecting it to a vertex of $K ^{+}_{3}$ with degree one, two and three respectively. 
	 For all such $F$,  we can get $P^2_{7}$ packs with $F$ by $P^2_{7}$ packs with $K_3$.

	Let $n = 8$. Then $t = 2$. The unique graph $H$ with $\delta(H) \geqslant  3$ and $e(H) \leqslant 6$ is $K_4$. 
Since $F$ is $\mathcal{H}^*_{8}$-free graph and  $K_4 \in\mathcal{H}^*_{8}$, thus there is no graph in case (a). 
	Note that after deleting a vertex with degree at most two, the graphs in case (c) must contain $K^-_4$ as a subgraph.
	Since there are at most 6 edges in $F$ and $F$ is $K_4$-free graph, thus there is no graph in case (c).
	
	Let $n = 9$. Then $t = 2$. 
	 The unique graph $H$ with $\delta(H)\geqslant 3$ and $e(H) \leqslant 7$ is $K_4$. 
	Since $F$ is $\mathcal{H}^*_{9}$-free graph and  $K_4 \in\mathcal{H}^*_{9}$, 
	 there is no graph in case (a).	 
	Since $\overline{P^2_{9}}$ is $K_4$-free graph (the three vertices of each triangle of $\overline{P^2_{9}}$ have no common neighbors, see Figure 3), there
	is no graph in case (c).
	
	Let $n = 10$. Then $t = 2$. The graphs in case (a) are $K_4$ and $W_5$
	(the graph obtained from $C_4$ by adding a new
	vertex and joining it to all vertices of $C_4$).
%
 We can easy get that $F$ packs with $K_4$ and $W_5$ (see Figure 3). 
	The graphs in case (c) are $K ^+_4 \cup K_2$, $K_4 \cup M_2$, $K_4 \cup P_3$, $G_4$, $G_5$, $G_6$ and $G_7$, where $G_4$, $G_5$ and $G_6$ are obtained from $K ^+_4$ by adding a new vertex and joining it to a vertex of $K ^+_4$ with degree one, three and four respectively and $G_7$ is obtained from $K_4$ by adding
	an isolated vertex and joining it to two vertices of $K_4$. 
%
	For all such $F$,  we can get $P^2_{10}$ packs with $F$ by $P^2_{10}$ packs with $K_4$ (see Figure 4).

	Let $n = 11$. Then $t = 2$. In case (a) the graphs with minimum degree at least three and on at most 9 edges are $K_4$, $W_5$,
	$K ^- _5$, $K_{3,3}$ and $G_8$, where $K_{3,3}$ is the complete bipartite graph with partite sets with sizes 3 and 3, and $G_8$ is obtained from two vertex disjoint copies of $K_3$ and joining three independent edges between them. Obviously, $P^2_{11}$	packs with each graph in case (a) (see Figures 5 and 6). Clearly, there is no graph in case (c).
	
	Let $n = 12$. Then $t = 3$. 
	The unique graph $H$ with $\delta(H) \geqslant  4$ and $e(H) \leqslant 10$ is $K_5$.
	Since $F$ is $\mathcal{H}^*_{12}$-free graph, thus there is no graph in case(a). 
	 Clearly, there is no graph in case (c).

	Let $n = 13$. Then $t = 3$. In case (a) the unique graph with minimum degree at least 4 on at most 11 edges is $K_5$. 
	It is obvious that $P^2_{13}$ packs with $K_5$. 
	Now the graphs in case (c) are $K^ +_5$ and $K_5\cup K_2$. 
	Since $P^2_{13}$ packs with $K_5$ (see Figure 6),  $P^2_{13}$ packs with $K^+ _5$ and $K_5 \cup K_2$.

	Suppose it is true for $n- 1 \geqslant 13$. For each graph
	on at most $n -3$ edges, there is a graph in $\mathcal{K}(n ,n-2)\setminus \{S_{n-1},S_{n-2}\cup K_2\}$ contains it as a subgraph. 
	It is sufficient to show that $P^2_{n}$
	packs with each $F\in \mathcal{K}(n,n-2)\setminus \{S_{n-1},S_{n-2}\cup K_2\}$. 
	Then by induction hypothesis,  $P^2_{n-1}$ packs with each $F'\in \mathcal{K}(n-1,n-3)\setminus \{S_{n-2},S_{n-3}\cup K_2\}$. 
	We consider the following two cases. 
	(a). $1\leqslant\delta^*(F)\leqslant t$. By Proposition \ref{proposition2.0}, we get that $P^2_{n}$ packs with $F$.
	 (b). $\delta^*(F)\geqslant t+1$. Then the number of non-isolated vertices of $ F$ is at most 
	 $\lfloor 2(n -2)/ \lceil(n+4)/4\rceil \rfloor$. 
	 On the other hand, it is easy to
	see that $P^2_{n}$ packs with $K_s$, 
	where $s =\lceil{n}/{3}\rceil$. 
	If $n \geqslant 16$, then we have 
	$\lfloor 2(n -2)/  \lceil(n+4)/4\rceil  \rfloor\leqslant\lceil{n}/{3}\rceil$, i.e., $K_s$ contains $F$.
	Thus $P^2_{n}$ packs with $F$. Let
	$n\in\{14, 15\}$.
Then $t=3$. 
By consider the neighbors of $P^2_{n}$, $P^2_{n}$ packs with $K^- _6$.
 Since $F$ has at most $n-2\leqslant 13$ edges and $\delta^\ast(F)\geqslant 4$, the number of non-isolated
	vertices of $F$ is at most  6, whence $K^-_6$  contains $F$.
Therefore, $P^2_{n}$ packs with $F$, the proof is complete.
	\hfill$\square$ \medskip

	\begin{center}
		\begin{tikzpicture}
			
			\path
			(0, 0)
			\foreach \i in {0, ..., 10} {
				+(360/11 * \i:1cm) coordinate (corner \i)
			}
			;\draw [line width=0.2pt]
			(corner 0) -- (corner 3)
			(corner 0) -- (corner 4)
			(corner 0) -- (corner 5)
			(corner 0) -- (corner 6)
			(corner 0) -- (corner 7)
			(corner 0) -- (corner 8)
			(corner 1) -- (corner 4)
			(corner 1) -- (corner 5)
			(corner 1) -- (corner 6)
			(corner 1) -- (corner 7)
			(corner 1) -- (corner 8)
			(corner 1) -- (corner 9)
			(corner 2) -- (corner 5)
			(corner 2) -- (corner 6)
			(corner 2) -- (corner 7)
			(corner 2) -- (corner 8)
			(corner 2) -- (corner 9)
			(corner 2) -- (corner 10)
			(corner 3) -- (corner 6)
			(corner 3) -- (corner 7)
			(corner 3) -- (corner 8)
			(corner 3) -- (corner 9)
			(corner 3) -- (corner 10)
			(corner 4) -- (corner 9)
			(corner 4) -- (corner 7)
			(corner 4) -- (corner 8)
			(corner 4) -- (corner 10)
			(corner 5) -- (corner 8)
			(corner 5) -- (corner 9)
			(corner 5) -- (corner 10)
			(corner 6) -- (corner 9)
			(corner 6) -- (corner 10)
			(corner 7) -- (corner 10)
			(corner 7) -- (corner 9)
			(corner 8) -- (corner 10)
			;
			\draw [line width=1pt][red]
			(corner 1) -- (corner 9)
			(corner 1) -- (corner 8)
			(corner 1) -- (corner 4)
			(corner 4) -- (corner 8)
			(corner 4) -- (corner 9)
			(corner 9) -- (corner 8)
			;
			
			\fill[radius=2pt] \foreach \i in {0, ..., 10} { (corner \i) circle(2pt) };
			\node at (0,-1.5){$\overline{P_{11}^2}$};
			
			\path
			(3, 0)
			\foreach \i in {0, ..., 10} {
				+(360/11 * \i:1cm) coordinate (corner \i)
			}
			;\draw [line width=0.2pt]
			(corner 0) -- (corner 3)
			(corner 0) -- (corner 4)
			(corner 0) -- (corner 5)
			(corner 0) -- (corner 6)
			(corner 0) -- (corner 7)
			(corner 0) -- (corner 8)
			(corner 1) -- (corner 4)
			(corner 1) -- (corner 5)
			(corner 1) -- (corner 6)
			(corner 1) -- (corner 7)
			(corner 1) -- (corner 8)
			(corner 1) -- (corner 9)
			(corner 2) -- (corner 5)
			(corner 2) -- (corner 6)
			(corner 2) -- (corner 7)
			(corner 2) -- (corner 8)
			(corner 2) -- (corner 9)
			(corner 2) -- (corner 10)
			(corner 3) -- (corner 6)
			(corner 3) -- (corner 7)
			(corner 3) -- (corner 8)
			(corner 3) -- (corner 9)
			(corner 3) -- (corner 10)
			(corner 4) -- (corner 9)
			(corner 4) -- (corner 7)
			(corner 4) -- (corner 8)
			(corner 4) -- (corner 10)
			(corner 5) -- (corner 8)
			(corner 5) -- (corner 9)
			(corner 5) -- (corner 10)
			(corner 6) -- (corner 9)
			(corner 6) -- (corner 10)
			(corner 7) -- (corner 10)
			(corner 7) -- (corner 9)
			(corner 8) -- (corner 10)
			;
			;
			\draw [line width=1pt][red]
			(corner 8) -- (corner 10)
			(corner 3) -- (corner 10)
			(corner 6) -- (corner 10)
			(corner 9) -- (corner 8)
			(corner 6) -- (corner 9)
			(corner 3) -- (corner 9)
			(corner 8) -- (corner 3)
			(corner 3) -- (corner 6)
			
			;
			\fill[radius=2pt] \foreach \i in {0, ..., 10} { (corner \i) circle(2pt) };
			\node at (3,-1.5){$\overline{P_{11}^2}$};

			\path
			(6, 0)
			\foreach \i in {0, ..., 10} {
				+(360/11 * \i:1cm) coordinate (corner \i)
			}
			;\draw [line width=0.2pt]
			(corner 0) -- (corner 3)
			(corner 0) -- (corner 4)
			(corner 0) -- (corner 5)
			(corner 0) -- (corner 6)
			(corner 0) -- (corner 7)
			(corner 0) -- (corner 8)
			(corner 1) -- (corner 4)
			(corner 1) -- (corner 5)
			(corner 1) -- (corner 6)
			(corner 1) -- (corner 7)
			(corner 1) -- (corner 8)
			(corner 1) -- (corner 9)
			(corner 2) -- (corner 5)
			(corner 2) -- (corner 6)
			(corner 2) -- (corner 7)
			(corner 2) -- (corner 8)
			(corner 2) -- (corner 9)
			(corner 2) -- (corner 10)
			(corner 3) -- (corner 6)
			(corner 3) -- (corner 7)
			(corner 3) -- (corner 8)
			(corner 3) -- (corner 9)
			(corner 3) -- (corner 10)
			(corner 4) -- (corner 9)
			(corner 4) -- (corner 7)
			(corner 4) -- (corner 8)
			(corner 4) -- (corner 10)
			(corner 5) -- (corner 8)
			(corner 5) -- (corner 9)
			(corner 5) -- (corner 10)
			(corner 6) -- (corner 9)
			(corner 6) -- (corner 10)
			(corner 7) -- (corner 10)
			(corner 7) -- (corner 9)
			(corner 8) -- (corner 10)
			;
			\draw [line width=1pt][red]
			(corner 1) -- (corner 8)
			(corner 4) -- (corner 7)
			(corner 1) -- (corner 7)
			(corner 1) -- (corner 9)
			(corner 4) -- (corner 9)
			(corner 8) -- (corner 9)
			(corner 7) -- (corner 9)
			(corner 8) -- (corner 4)
			(corner 1) -- (corner 4)
			;

			\fill[radius=2pt] \foreach \i in {0, ..., 10} { (corner \i) circle(2pt) };
			\node at (6,-1.5){$\overline{P_{11}^2}$};
			
			\path
			(9, 0)
			\foreach \i in {0, ..., 10} {
				+(360/11 * \i:1cm) coordinate (corner \i)
			}
			;\draw [line width=0.2pt]
			(corner 0) -- (corner 3)
			(corner 0) -- (corner 4)
			(corner 0) -- (corner 5)
			(corner 0) -- (corner 6)
			(corner 0) -- (corner 7)
			(corner 0) -- (corner 8)
			(corner 1) -- (corner 4)
			(corner 1) -- (corner 5)
			(corner 1) -- (corner 6)
			(corner 1) -- (corner 7)
			(corner 1) -- (corner 8)
			(corner 1) -- (corner 9)
			(corner 2) -- (corner 5)
			(corner 2) -- (corner 6)
			(corner 2) -- (corner 7)
			(corner 2) -- (corner 8)
			(corner 2) -- (corner 9)
			(corner 2) -- (corner 10)
			(corner 3) -- (corner 6)
			(corner 3) -- (corner 7)
			(corner 3) -- (corner 8)
			(corner 3) -- (corner 9)
			(corner 3) -- (corner 10)
			(corner 4) -- (corner 9)
			(corner 4) -- (corner 7)
			(corner 4) -- (corner 8)
			(corner 4) -- (corner 10)
			(corner 5) -- (corner 8)
			(corner 5) -- (corner 9)
			(corner 5) -- (corner 10)
			(corner 6) -- (corner 9)
			(corner 6) -- (corner 10)
			(corner 7) -- (corner 10)
			(corner 7) -- (corner 9)
			(corner 8) -- (corner 10)
			(corner 8) -- (corner 9)
			;
			\draw [line width=1pt][red]
			(corner 0) -- (corner 4)
			(corner 0) -- (corner 5)
			(corner 0) -- (corner 6)
			(corner 1) -- (corner 4)
			(corner 1) -- (corner 5)
			(corner 1) -- (corner 6)
			(corner 10) -- (corner 4)
			(corner 10) -- (corner 5)
			(corner 10) -- (corner 6)
			;
			\fill[radius=2pt] \foreach \i in {0, ..., 10} { (corner \i) circle(2pt) };
			\node at (9,-1.5){$\overline{P_{11}^2}$};
			\node at (4.5,-2.5){Figure 5};
		\end{tikzpicture}
	\end{center}
	
\begin{center}
	\begin{tikzpicture}

		\path
		(0, 0)
		\foreach \i in {0, ..., 10} {
			+(360/11 * \i:1cm) coordinate (corner \i)
		}
		;\draw [line width=0.2pt]
		(corner 0) -- (corner 3)
		(corner 0) -- (corner 4)
		(corner 0) -- (corner 5)
		(corner 0) -- (corner 6)
		(corner 0) -- (corner 7)
		(corner 0) -- (corner 8)
		(corner 1) -- (corner 4)
		(corner 1) -- (corner 5)
		(corner 1) -- (corner 6)
		(corner 1) -- (corner 7)
		(corner 1) -- (corner 8)
		(corner 1) -- (corner 9)
		(corner 2) -- (corner 5)
		(corner 2) -- (corner 6)
		(corner 2) -- (corner 7)
		(corner 2) -- (corner 8)
		(corner 2) -- (corner 9)
		(corner 2) -- (corner 10)
		(corner 3) -- (corner 6)
		(corner 3) -- (corner 7)
		(corner 3) -- (corner 8)
		(corner 3) -- (corner 9)
		(corner 3) -- (corner 10)
		(corner 4) -- (corner 9)
		(corner 4) -- (corner 7)
		(corner 4) -- (corner 8)
		(corner 4) -- (corner 10)
		(corner 5) -- (corner 8)
		(corner 5) -- (corner 9)
		(corner 5) -- (corner 10)
		(corner 6) -- (corner 9)
		(corner 6) -- (corner 10)
		(corner 7) -- (corner 10)
		(corner 7) -- (corner 9)
		(corner 8) -- (corner 10)
		(corner 8) -- (corner 9)
		;
		\draw [line width=1pt][red]
		(corner 0) -- (corner 8)
		(corner 0) -- (corner 4)
		(corner 0) -- (corner 5)
		(corner 1) -- (corner 8)
		(corner 1) -- (corner 5)
		(corner 1) -- (corner 9)
		(corner 9) -- (corner 4)
		(corner 9) -- (corner 5)
		(corner 9) -- (corner 1)
		(corner 8) -- (corner 4)
		;
		\fill[radius=2pt] \foreach \i in {0, ..., 10} { (corner \i) circle(2pt) };
		\node at (0,-1.5){$\overline{P_{11}^2}$};
		
		\path
		(5, 0)
		\foreach \i in {0, ..., 12} {
			+(360/13 * \i:1cm) coordinate (corner \i)
		}
		;\draw [line width=0.2pt]
		(corner 0) -- (corner 3)
		(corner 0) -- (corner 4)
		(corner 0) -- (corner 5)
		(corner 0) -- (corner 6)
		(corner 0) -- (corner 7)
		(corner 0) -- (corner 8)
		(corner 0) -- (corner 9)
		(corner 0) -- (corner 10)
		(corner 1) -- (corner 4)
		(corner 1) -- (corner 5)
		(corner 1) -- (corner 6)
		(corner 1) -- (corner 7)
		(corner 1) -- (corner 8)
		(corner 1) -- (corner 9)
		(corner 1) -- (corner 10)
		(corner 1) -- (corner 11)
		(corner 2) -- (corner 5)
		(corner 2) -- (corner 6)
		(corner 2) -- (corner 7)
		(corner 2) -- (corner 8)
		(corner 2) -- (corner 9)
		(corner 2) -- (corner 10)
		(corner 2) -- (corner 11)
		(corner 2) -- (corner 12)
		(corner 3) -- (corner 6)
		(corner 3) -- (corner 7)
		(corner 3) -- (corner 8)
		(corner 3) -- (corner 9)
		(corner 3) -- (corner 10)
		(corner 3) -- (corner 11)
		(corner 3) -- (corner 12)
		(corner 4) -- (corner 9)
		(corner 4) -- (corner 7)
		(corner 4) -- (corner 8)
		(corner 4) -- (corner 10)
		(corner 4) -- (corner 11)
		(corner 4) -- (corner 12)
		(corner 5) -- (corner 8)
		(corner 5) -- (corner 9)
		(corner 5) -- (corner 10)
		(corner 5) -- (corner 11)
		(corner 5) -- (corner 12)
		(corner 6) -- (corner 9)
		(corner 6) -- (corner 10)
		(corner 6) -- (corner 11)
		(corner 6) -- (corner 12)
		(corner 7) -- (corner 10)
		(corner 7) -- (corner 11)
		(corner 7) -- (corner 12)
		(corner 8) -- (corner 11)
		(corner 8) -- (corner 12)
		(corner 9) -- (corner 12)
		(corner 9) -- (corner 11)
		(corner 8) -- (corner 10)
		;
		\draw [line width=1pt][red]
		(corner 10) -- (corner 0)
		(corner 3) -- (corner 9)
		(corner 10) -- (corner 3)
		(corner 6) -- (corner 9)
		(corner 10) -- (corner 6)
		(corner 6) -- (corner 0)
		(corner 9) -- (corner 10)
		(corner 9) -- (corner 0)
		;
		\fill[radius=2pt] \foreach \i in {0, ..., 12} { (corner \i) circle(2pt) };
		\node at (5,-1.5){$\overline{P_{13}^2}$};
		\node at (2.5,-2.5){Figure 6};
	\end{tikzpicture}
\end{center}

	\section{Proof of Theorem \ref{theorem2}}\label{section 3}

	The proof of Theorem \ref{theorem2} is based on the following Lemmas.
	
	\begin{lemma}[Fiedler and Nikiforov \cite{2010Fiedler2170}\label{lemma3.1}]
		Let $G$ be a graph of order $n$ and spectral radius $\mu(G)$.
		If $$\mu(G)\geqslant n-2,$$ then $G$ contains a Hamiltonian path unless $G=K_{n-1}\cup K_1$.
	\end{lemma}
	
	\begin{lemma}[Hong \cite{1988Yuan135}\label{lemma3.2}]
		Let $G$ be a connected graph of order $n$ with $m$ edges.
		The spectral radius $\mu(G)$ satisfies 
		$\mu(G)\leqslant \sqrt{2m-n+1}$ with equality if and only if $G$ is isomorphic to $S_n$ or $K_n$.
	\end{lemma}
	
		\noindent {\bf Proof of Theorem~\ref{theorem2}.}
  Let $\mu(G)> n-2$ and $n\geqslant 6$. Suppose that $G$ is not a subgraph of $K_n - E(S_{n-1})$ and $K_6 - E(K_3)$ for $n\geqslant 6$.
   It follows from Lemma \ref{lemma3.1} that $G$ contains a Hamilton path ($K_n - E(S_{n-1})$ contains $K_{n-1}\cup K_1$ as a subgraph), whence $G$ is connected.
  By Lemma \ref{lemma3.2}, we have $\mu(G)\leqslant \sqrt{2e(G)-n+1}$,
  with equality if and only if $G=K_n$ ($S_n$ does not contain a Hamilton path).
Since $K_n$ contains a copy of $P_n^2$, we may assume that $\mu(G)< \sqrt{2e(G)-n+1}$.
Then $e(G)>(n^2-3n+3)/2$, implying $e(G)\geqslant  {n-1 \choose 2}+1$.
  By Theorem \ref{theorem1.1}, $G$ contains a copy of $P_n^2$ unless $G\in \{K_n- E(H) : H\in \mathcal{H}^*_n\}$.
  Since the maximum degree of $K_n - E(S_{n-2}\cup K_2)$ is $n-2$, we get  $\mu(K_n - E(S_{n-2}\cup K_2)) \leqslant n-2$.
  By tedious calculations, we get $\mu(K_6 - E( K_3)) > 4$, $\mu(K_7 - E(K_4^-)) < 5$, $\mu(K_n - E( K_4)) < n-2$ for $n=8,9$ and $\mu(K_{12} - E( K_5) ) <10$.
  Hence, the proof is complete.
	\hfill$\square$ \medskip

\bibliographystyle{elsarticle-num} 
  \bibliography{2_-power_of_Hamilton_paths-11-17}

\begin{thebibliography}{1}
\expandafter\ifx\csname url\endcsname\relax
  \def\url#1{\texttt{#1}}\fi
\expandafter\ifx\csname urlprefix\endcsname\relax\def\urlprefix{URL }\fi
\expandafter\ifx\csname href\endcsname\relax
  \def\href#1#2{#2} \def\path#1{#1}\fi

\bibitem{2022Khan112908}
Z.~U. Khan, L.~T. Yuan, A note on the 2-power of {H}amilton cycles, Discrete
  Mathematics 345~(8) (2022) 112908.

\bibitem{1961Ore315}
O.~Ore, Arc coverings of graphs, Ann. Mat. Pura Appl. (4) 55 (1961) 315--321.

\bibitem{2010Fiedler2170}
M.~Fiedler, V.~Nikiforov, Spectral radius and {H}amiltonicity of graphs, Linear
  Algebra and its Applications 432~(9) (2010) 2170--2173.

\bibitem{2023Yan113155}
X.~Yan, X.~He, L.~Feng, W.~Liu, Spectral radius and the 2-power of {H}amilton
  cycle, Discrete Mathematics 346~(1) (2023) 113155.

\bibitem{1988Yuan135}
H.~Yuan, A bound on the spectral radius of graphs, Linear Algebra and its
  Applications 108 (1988) 135--139.

\end{thebibliography}
\end{document}